\documentclass{amsart}

\usepackage{graphicx}
\usepackage{color}

\newtheorem{thm}{Theorem}

\newtheorem{lem}[thm]{Lemma}
\newtheorem{prop}[thm]{Proposition}
\newtheorem{cor}[thm]{Corollary}

\begin{document}
\title{A criterion for the Legendrian simplicity of the connected sum}
\author{Byung hee An}
\thanks{This work was supported by IBS-R003-D1}
\address{Center for Geometry and Physics, Institue for Basic Science (IBS), Pohang 790-784, Republic of Korea}
\email{anbyhee@ibs.re.kr}
\begin{abstract}
In this paper, we provide the necessary and sufficient conditions for the connected sum of knots in $S^3$ to be Legendrian simple.
\end{abstract}
\maketitle

\section{Introduction}

Throughout this paper, we consider oriented Legendrian knots in $S^3$ with the standard tight contact structure $\xi_{std}$.

A Legendrian knot theory is the study of Legendrian knots up to Legendrian isotopy which is much more restricted compared to the smooth or piecewise-linear isotopy in the classical knot theory. More precisely, by considering the ambient space as the one-point compactification of $\mathbb{R}^3$ and taking a suitable projection, these two theories are differ by whether {\em positive and negative (de)stabilizations} are allowed or not. Therefore for given topological knot type $K$, there are infinitely many inequivalent Lengendrian isotopy classes of topological knot type $K$, and we denote by $\mathcal{L}(K)$ the set of Legendrian isotopy classes of $K$.
There are two well-known Legendrian knot invariants, {\em Thurston-Bennequin number $tb(L)$} and {\em rotation number $r(L)$} that can be used to classify Legendrian knots in $\mathcal{L}(K)$, and they are called {\em classical invariants}. Refer to \cite{Et} for details.
Then a topological knot type $K$ is said to be {\em Legendrian simple} if Legendrian knots in $\mathcal{L}(K)$ are classified by the classical invariants.

The Legendrian simplicity for unknot has been shown by Eliashberg and Fraser in 1995 \cite{EF}, and for figure-eight and torus knots by Etnyre and Honda in 2001 \cite{EH1}. Especially, Etnyre and Honda in their followed paper \cite{EH2} provided a complete combinatorial description for the connected sum of Legendrian knots, and proved that the Legendrian simplicity is not closed under the connected sum. 

On the other hand, they also proved in \cite{EH3} that the cabling operation preserves the Legendrian simplicity under the {\em uniform thickness property} (UTP). The important benifits of (UTP) are that it will be preserved not only by cabling with the certain condition, but also by the connected sum. Hence, if started with Legendrian simple and UTP knot, one may produce the arbitrarily many Legendrian simple knots by the iterated cabling as studied in \cite{ELT, La1, La2}.
However, as seen above, the Legendrian simplicity may not be preserved under connected sum, and therefore the attempt of cabling after connected sum may fail.

This paper provides the necessary and sufficient conditions for the connected sum to be Legendrian simple as follows.

\begin{thm}\label{thm:criterion}
Let $K=(\#^{a_1}K_1)\#\dots(\#^{a_n}K_n)$ be a connected sum decomposition of $K$ with pairwise distinct prime knots. Then $K$ is Legendrian simple if and only if 
all $K_i$'s are Legendrian simple and one of the following is satisfied:
\begin{enumerate}
\item $|Peak(K_i)|=1$ for all $i$;
\item There exists only one $i$ such that $|Peak(K_i)|=2$, $a_i\ge 2$ and $|Peak(K_j)|=1$ if $j\neq i$;
\item There exists only one $i$ such that $|Peak(K_i)|\ge 3$, $a_i=1$ and $|Peak(K_j)|=1$ if $j\neq i$.
\end{enumerate}
Here, $Peak(K_i)$ is the set of Legendrian knots in $\mathcal{L}(K_i)$ which can not be destabilized in either ways.
\end{thm}

Therefore this theorem provides the way to produce the new Legendrian simple knots. Moreover, this can be used with the the cabling construction, as well.

The rest of this paper consists of the following. In section 2, we introduce the basic notions and briefly review the known results about Legendrian connected sum. In section 3, we prove the main result.

\section{preliminaries}

\subsection{Basic notions}
A {\em topological oriented knot}, simply a {\em knot $K$} from now on, is a smooth embedding of $S^1$ into $S^3$, and a {\em knot type $[K]$} is a smooth isotopy class of $K$. The natural orientation of $K$ comes from $d\theta$ where $S^1$ is parametrized by $\theta$.
A knot $K$ is {\em trivial} or an {\em unknot} if $K$ bounds an embedded disc in $S^3$.

A {\em Legendrian knot $L$} is a Legendrian embedding of $S^1$ into $(S^3,\xi_{std})$, that is, $L$ is everywhere tangent to the contact plane $\xi_{std}$. Here $\xi_{std}$ is the standard contact structure on $S^3$. Similarly, a {\em Legendrian knot type $[L]$} is the Legendrian isotopy class of $L$.
From now on, we simply use $K$, $L$ instead of $[K]$, $[L]$ to denote knot types unless any ambiguity occurs.

For given knot type $K$, we denote by $\mathcal{L}(K)$ the set of Legendrian knots of topological knot type $K$. Then as mentioned before, any two elements in $\mathcal{L}(K)$ can be connected with a sequence of two special types of isotopies, called {\em positive and negative stabilizations $S_{\pm}$}, and their inverses. Note that these two stabilizations are commutative. In other words, $S_+(S_-(L))=S_-(S_+(L))$ for any Legendrian knot $L$.
In the diagrammatic viewpoint, $S_{\pm}$ is as depicted in Figure~\ref{fig:stabilization}. Hence we can regard $\mathcal{L}(K)$ as a connected, directed graph by adding directed edges $(L, S_+(L))$ and $(L,S_{-}(L))$ for each $L\in\mathcal{L}(K)$.

\begin{figure}[ht]
\begingroup%
  \makeatletter%
  \providecommand\color[2][]{%
    \errmessage{(Inkscape) Color is used for the text in Inkscape, but the package 'color.sty' is not loaded}%
    \renewcommand\color[2][]{}%
  }%
  \providecommand\transparent[1]{%
    \errmessage{(Inkscape) Transparency is used (non-zero) for the text in Inkscape, but the package 'transparent.sty' is not loaded}%
    \renewcommand\transparent[1]{}%
  }%
  \providecommand\rotatebox[2]{#2}%
  \ifx\svgwidth\undefined%
    \setlength{\unitlength}{176bp}%
    \ifx\svgscale\undefined%
      \relax%
    \else%
      \setlength{\unitlength}{\unitlength * \real{\svgscale}}%
    \fi%
  \else%
    \setlength{\unitlength}{\svgwidth}%
  \fi%
  \global\let\svgwidth\undefined%
  \global\let\svgscale\undefined%
  \makeatother%
  \begin{picture}(1,0.28711829)%
    \put(0,0){\includegraphics[width=\unitlength]{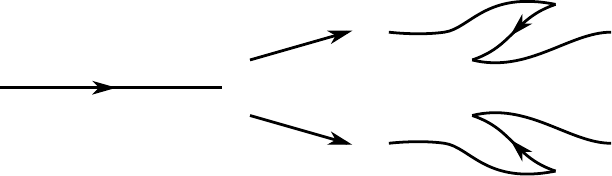}}%
    \put(0.40909091,0.23446825){\color[rgb]{0,0,0}\makebox(0,0)[lb]{\smash{$S_+$}}}%
    \put(0.40909091,0.02992279){\color[rgb]{0,0,0}\makebox(0,0)[lb]{\smash{$S_-$}}}%
  \end{picture}%
\endgroup%
\caption{Positive and negative stabilization $S_{\pm}$ in the front projection}
\label{fig:stabilization}
\end{figure}

The classical invariants, Thurston-Bennequin number $tb(L)$ and the rotation number $r(L)$, for a Legendrian knot $L$ 
change as follows under stabilizations.
$$
tb(S_{\pm}(L)) = tb(L) -1,\quad r(S_{\pm}(L))=r(L)\pm 1.
$$
This implies that $\mathcal{L}(K)$ has no directed loop different from constant, and admits a graded poset structure $(\mathcal{L}(K),\prec)$ by declaring $S_{\pm}(L)\prec L$, whose grading is given by $tb$.
We say that $L$ is a {\em parent} of both $S_+(L)$ and $S_-(L)$. 

On the other hand, there is another important poset $\mathcal{M}(K)$, called {\em mountain range} defined by the range $\Phi(\mathcal{L}(K))$ of the pair $\Phi=(tb,r):\mathcal{L}(K)\to\mathbb{Z}^2$ of classical invariants.
We will follow the common convention of regarding $tb$ and $r$ as vertical and horizontal axes, respectively, and the words, such as {\em left, right} and so on, have the suitable meaning according to this convention. Especially, {\em maximal} means having maximal $tb$ among elements of given (subset of) $\mathcal{L}(K)$ or $\mathcal{M}(K)$.

Note that $\mathcal{M}(K)$ is never bounded below because $tb$ can be decreased arbitrarily by taking stabilizations in $\mathcal{L}(K)$. However, Bennequin in \cite{Be} showed that $\mathcal{M}(K)$ is always bounded above as follows.

\begin{thm}\cite{Be}\label{thm:Benn}
For given knot type $K$ and $L\in\mathcal{L}(K)$,
$$tb(L)+ |r(L)|\le 2 g(K)-1,$$
where $g(K)$ is a {\em genus} of $K$.
\end{thm}

This result has been improved in many ways, related to classical link invariants such as genus \cite{Be, Ru2}, polynomials \cite{Mo,Ru1}, or other invariants such as Khovanov homology \cite{Ng}, knot Floer homology \cite{Pl} and so on.
%
However they are not essential in this paper and we omit the detail.
%
%

Though $\Phi$ is not injective in general, 
we draw $\mathcal{L}(K)$ on the $(tb,r)$-plane via $\Phi$ by perturbing edges and vertices slightly if necessary, as seen in Figure~\ref{fig:peakvalley}.

\begin{figure}[ht]
\begingroup%
  \makeatletter%
  \providecommand\color[2][]{%
    \errmessage{(Inkscape) Color is used for the text in Inkscape, but the package 'color.sty' is not loaded}%
    \renewcommand\color[2][]{}%
  }%
  \providecommand\transparent[1]{%
    \errmessage{(Inkscape) Transparency is used (non-zero) for the text in Inkscape, but the package 'transparent.sty' is not loaded}%
    \renewcommand\transparent[1]{}%
  }%
  \providecommand\rotatebox[2]{#2}%
  \ifx\svgwidth\undefined%
    \setlength{\unitlength}{252.14074313bp}%
    \ifx\svgscale\undefined%
      \relax%
    \else%
      \setlength{\unitlength}{\unitlength * \real{\svgscale}}%
    \fi%
  \else%
    \setlength{\unitlength}{\svgwidth}%
  \fi%
  \global\let\svgwidth\undefined%
  \global\let\svgscale\undefined%
  \makeatother%
  \begin{picture}(1,0.44677759)%
    \put(0,0){\includegraphics[width=\unitlength]{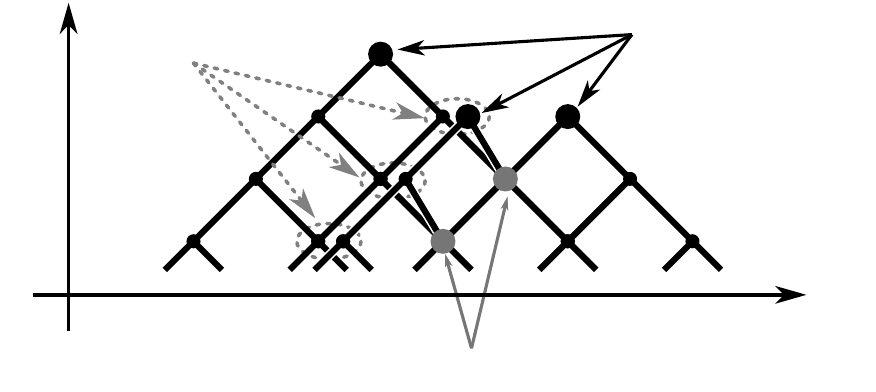}}%
    \put(0.72996082,0.41107544){\color[rgb]{0,0,0}\makebox(0,0)[lb]{\smash{peaks}}}%
    \put(-0.00308218,0.41543872){\color[rgb]{0,0,0}\makebox(0,0)[lb]{\smash{$tb$}}}%
    \put(0.93358389,0.11000416){\color[rgb]{0,0,0}\makebox(0,0)[lb]{\smash{$r$}}}%
    \put(0.11909167,0.382859){\color[rgb]{0,0,0}\makebox(0,0)[lb]{\smash{nonsimple}}}%
    \put(0.52226529,0.00819265){\color[rgb]{0,0,0}\makebox(0,0)[lb]{\smash{valleys}}}%
  \end{picture}%
\endgroup%
\caption{$Peak$ and $Valley$ for a poset}
\label{fig:peakvalley}
\end{figure}
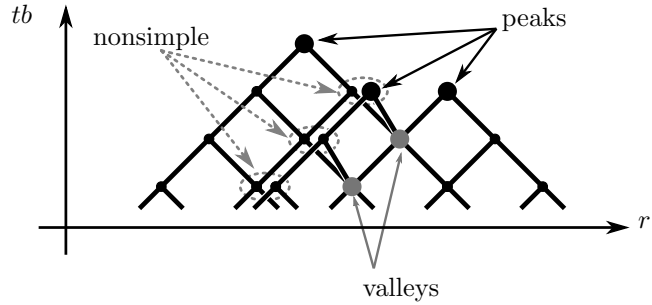

Then we say that $K$ is {\em Legendrian simple} if $\Phi$ is injective.
The simplicity can be defined pointwise as well, according to the number of the inverse image under $\Phi$. In other words, a point $p\in\mathcal{M}(K)$ is {\em simple} if $p$ has only one inverse under $\Phi$, and $p$ is {\em nonsimple} otherwise.

For a given poset $\mathcal{P}$, we denote by $Peak(\mathcal{P})$ the set of elements of $\mathcal{P}$ having no parent, and by $Valley(\mathcal{P})$ the set of elements $V\in\mathcal{P}$ such that $V$ possesses two parents which have no common parent. A possible example of $Peak$ and $Valley$ for a poset is depicted in Figure~\ref{fig:peakvalley}.
We will consider $Peak(\mathcal{L}(K))$, $Valley(\mathcal{L}(K))$ as well as $Peak(\mathcal{M}(K))$, $Valley(\mathcal{M}(K))$.

If $K$ is Legendrian simple, then $\mathcal{L}(K)$ and $\mathcal{M}(K)$ are isomorphic as posets, and we consider $Peak$ and $Valley$ only for $\mathcal{L}(K)$ and denote by $Peak(K)$ and $Valley(K)$.
One of the obvious observation for a Legendrian simple knot $K$ is that 
$$|Peak(K)|=|Valley(K)|+1<\infty.$$

On the other hand, if $K$ is not Legendrian simple, then we can choose a point $N_{max}$ which is maximal among nonsimple points in $\mathcal{M}(K)$.
Observe that $N_{max}$ may not be unique, and any point above $N_{max}$ in $\mathcal{M}(K)$ is simple by definition.
\begin{lem}\label{lem:nmax}
Let $K$ be a Legendrian nonsimple knot and $N_{max}$ be as above.
Then either 
\begin{enumerate}
\item $\Phi^{-1}(N_{max})\cap Peak(\mathcal{L}(K))\neq\emptyset$ or 
\item $|\Phi^{-1}(N_{max})|=2$ and $N_{max}\in Valley(\mathcal{M}(K))$.
\end{enumerate}
\end{lem}
\begin{proof}
Let $\Phi^{-1}(N_{max})=\{L_1,\dots,L_k\}$. Since $N_{max}$ is not simple, $k\ge 2$.

Suppose $L_i\not\in Peak(\mathcal{L}(K))$ for all $i$. 
Then all $L_i$'s have at least 1 parent in $\mathcal{L}(K)$ and all these parents become parents of $N_{max}$ in $\mathcal{M}(K)$ via $\Phi$.
Since $N_{max}$ has at most 2 parents in $\mathcal{M}(K)$ and all points above $N_{max}$ are simple, there are at most 2 Legendrian knots which are parents of $L_i$'s. 
Hence $\Phi^{-1}(N_{max})$ consists of exactly two Legendrian knots where each has only 1 parent, and therefore $N_{max}\in Valley(\mathcal{M}(K)$. 
\end{proof}

Examples of possible local pictures near $N_{max}$ are depicted in Figure~\ref{fig:nmax}. The left two are involving at least 1 peak but the right one corresponds to a valley in $\mathcal{M}(K)$.

\begin{figure}[ht]
\begingroup%
  \makeatletter%
  \providecommand\color[2][]{%
    \errmessage{(Inkscape) Color is used for the text in Inkscape, but the package 'color.sty' is not loaded}%
    \renewcommand\color[2][]{}%
  }%
  \providecommand\transparent[1]{%
    \errmessage{(Inkscape) Transparency is used (non-zero) for the text in Inkscape, but the package 'transparent.sty' is not loaded}%
    \renewcommand\transparent[1]{}%
  }%
  \providecommand\rotatebox[2]{#2}%
  \ifx\svgwidth\undefined%
    \setlength{\unitlength}{204.15234375bp}%
    \ifx\svgscale\undefined%
      \relax%
    \else%
      \setlength{\unitlength}{\unitlength * \real{\svgscale}}%
    \fi%
  \else%
    \setlength{\unitlength}{\svgwidth}%
  \fi%
  \global\let\svgwidth\undefined%
  \global\let\svgscale\undefined%
  \makeatother%
  \begin{picture}(1,0.24534805)%
    \put(0,0){\includegraphics[width=\unitlength]{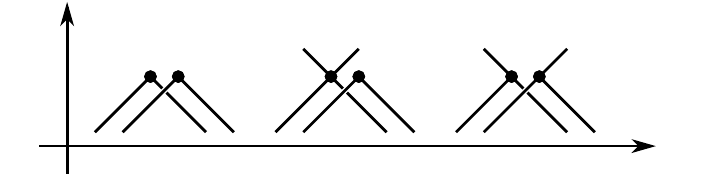}}%
    \put(-0.00325278,0.21557555){\color[rgb]{0,0,0}\makebox(0,0)[lb]{\smash{$tb$}}}%
    \put(0.93722136,0.03923665){\color[rgb]{0,0,0}\makebox(0,0)[lb]{\smash{$r$}}}%
  \end{picture}%
\endgroup%
\caption{Local pictures of $\mathcal{L}(K)$ near $N_{max}$}
\label{fig:nmax}
\end{figure}

\subsection{Connected sums}
Let $K_1, K_2$ be two knots in $S^3$. Roughly speaking, a {\em connected sum} $K_1\#K_2\subset S^3$ is a knot obtained by gluing $K_1\setminus \alpha_1$ and $K_2\setminus \alpha_2$ along the oriented end points, where $\alpha_i$'s are subarcs of $K_i$ unknotted in small enough 3 balls $B_i\subset S^3$.
It is equivalent to the connected sum of pairs $(S^3, K_1)$ and $(S^3, K_2)$.
For convenience sake, we denote by $\#^nK$ the connected sum of $n$-copies of $K$.

A nontrivial knot $K$ is said to be {\em prime} if $K=K_1\#K_2$ implies that one of $K_i$'s is trivial. If $K$ is neither trivial nor prime, we say that $K$ is {\em composite}.
Then for any nontrivial knot $K$, there is a unique prime decomposition $K=(\#^{a_1}K_1)\#\dots(\#^{a_n}K_n)$ up to permutting $K_i$'s, where $a_i\ge 1$ and all $K_i$'s are pairwise different prime knots.
For two knots $K_1$ and $K_2$,
we say that $K_1$ and $K_2$ are {\em relatively prime} unless there is nontrivial common connected summand.

To make this concepts fit into the Legendrian knot theory, we need the following.

\begin{thm}\cite{Co} Given two 3-manifolds $M_1, M_2$
there is an isomorphism
$$
\pi_0(Tight(M_1))\times \pi_0(Tight(M_2))\stackrel{\sim}{\longrightarrow}\pi_0(Tight(M_1\#M_2)),
$$
where $\pi_0(Tight(M_i))$ is the set of contact structures on $M_i$ up to contact isotopy.
\end{thm}
Since we consider $S^3$ as the ambient space which has the unique contact structure $\xi_{std}$ up to contact isotopy, there is no ambiguity at all.
Therefore, we may define $L_1\#L_2$ in $(S^3,\xi_{std})$ by the connected sum of pairs as before for given two Legendrian knots $L_1, L_2$ in $(S^3,\xi_{std})$. Figure~\ref{fig:connectedsum} shows a pictorial definition of the connected sum of (Legendrian) knots.

\begin{figure}[ht]
\begingroup%
  \makeatletter%
  \providecommand\color[2][]{%
    \errmessage{(Inkscape) Color is used for the text in Inkscape, but the package 'color.sty' is not loaded}%
    \renewcommand\color[2][]{}%
  }%
  \providecommand\transparent[1]{%
    \errmessage{(Inkscape) Transparency is used (non-zero) for the text in Inkscape, but the package 'transparent.sty' is not loaded}%
    \renewcommand\transparent[1]{}%
  }%
  \providecommand\rotatebox[2]{#2}%
  \ifx\svgwidth\undefined%
    \setlength{\unitlength}{316.80050441bp}%
    \ifx\svgscale\undefined%
      \relax%
    \else%
      \setlength{\unitlength}{\unitlength * \real{\svgscale}}%
    \fi%
  \else%
    \setlength{\unitlength}{\svgwidth}%
  \fi%
  \global\let\svgwidth\undefined%
  \global\let\svgscale\undefined%
  \makeatother%
  \begin{picture}(1,0.12880565)%
    \put(0,0){\includegraphics[width=\unitlength]{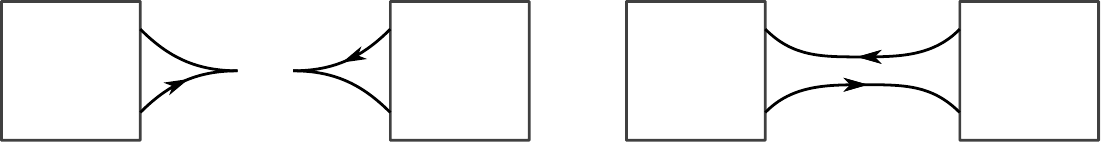}}%
    \put(0.05176915,0.05180243){\color[rgb]{0,0,0}\makebox(0,0)[lb]{\smash{$L_1$}}}%
    \put(0.40530394,0.04819494){\color[rgb]{0,0,0}\makebox(0,0)[lb]{\smash{$L_2$}}}%
    \put(0.22853654,0.05450806){\color[rgb]{0,0,0}\makebox(0,0)[lb]{\smash{$\#$}}}%
    \put(0.50901954,0.05090056){\color[rgb]{0,0,0}\makebox(0,0)[lb]{\smash{$=$}}}%
    \put(0.61995006,0.05180243){\color[rgb]{0,0,0}\makebox(0,0)[lb]{\smash{$L_1$}}}%
    \put(0.92297988,0.04819493){\color[rgb]{0,0,0}\makebox(0,0)[lb]{\smash{$L_2$}}}%
  \end{picture}%
\endgroup%
\caption{A pictorial definition of the connected sum}
\label{fig:connectedsum}
\end{figure}
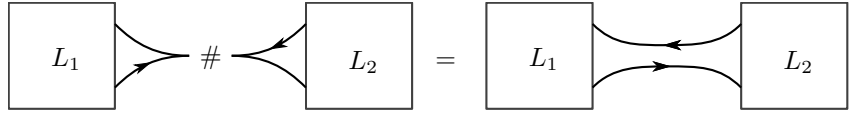

For given two knots $K_1, K_2$ and Legendrian knots $L_i\in\mathcal{L}(K_i)$, Etnyre and Honda in \cite{EH2} showed not only well-definedness of $L_1\#L_2$ but also a complete description of the relation between $\mathcal{L}(K_1\#K_2)$ and $\mathcal{L}(K_i)$'s for knots in arbitrary 3-manifolds with tight contact structures. Here we introduce their results only for knots in $S^3$. We denote by $\mathbf{S}_n$ {\em the symmetric group on $\{1,\dots,n\}$}.

\begin{thm}\cite{EH2}\label{thm:connectedsum}
Let $K = K_1\#K_2\#\dots\#K_n$ be a prime decomposition of a knot $K$ in $S^3$. Then the map
$$
C:\left(
\mathcal{L}(K_1)\times\dots\times\mathcal{L}(K_n)/\sim
\right)\to\mathcal{L}(K_1\#\dots\#K_n)
$$
given by $(L_1,\dots,L_n)\mapsto L_1\#\dots\#L_n$ is a bijection. Here the equivalence relation $\sim$ is of two types:
\begin{enumerate}
\item $(L_1,\dots,S_{\pm}(L_i),L_{i+1},\dots,L_n)\sim(L_1,\dots,L_i,S_{\pm}(L_{i+1}),\dots,L_n)$,
\item $(L_1,\dots,L_n)\sim(L_{\sigma(1)},\dots,L_{\sigma(n)})$, where $\sigma\in\mathbf{S}_n$ such that $K_{\sigma(i)}$ is isotopic to $K_i$.
\end{enumerate}
\end{thm}

The behavior of classical invariants under the connected sum is quite obvious as follows.
\begin{lem}\label{lem:classical}
Let $L_1, L_2$ be two Legendrian knots. Then
$$
tb(L_1\#L_2) = tb(L_1)+tb(L_2)+1,\quad
r(L_1\#L_2) = r(L_1) + r(L_2).
$$
\end{lem}

For a topological space $X$, a {\em symmetric product} $\operatorname{Sym}^n(X)$ of $X$ is defined by the quotient $\prod^n X/\mathbf{S}_n$ under the obvious action of the symmetric group.
We denote an equivalent class of $(x_1,\dots,x_n)$ (or a set with repetition) in $\operatorname{Sym}^n(X)$ by $[x_1,\dots,x_n]$.
Then the direct consequences of the above theorem are as follows.
\begin{cor}\label{cor:peaks1}
For a prime $K$, then 
$$
C:\operatorname{Sym}^n(Peak(\mathcal{L}(K))) \to Peak(\mathcal{L}(\#^n K))
$$
is bijective.
\end{cor}
\begin{proof}
It is obvious that type (2) equivalence relation is always applicable for any $\sigma\in\mathbf{S}_n$ but type (1) is never applicable on $\prod^n Peak(\mathcal{L}(K))$ because any element in $Peak(\mathcal{L}(K))$ can not be destabilized. Hence $C$ is well-defined on $\operatorname{Sym}^n(Peak(\mathcal{L}(K)))$. Moreover, by definition of the connected sum, $C$ maps bijectively onto $Peak(\mathcal{L}(\#^nK))$.
\end{proof}

\begin{cor}\label{cor:peaks2}
For two relatively prime knots $K_1$ and $K_2$,
$$
C:Peak(\mathcal{L}(K_1))\times Peak(\mathcal{L}(K_n))\to Peak(\mathcal{L}(K_1\#K_2))
$$
is bijective.
\end{cor}
\begin{proof}
Since $K_1$ and $K_2$ are relatively prime, the equivalence relation $\sim$ does not change two summand. Moreover, by the same reason as above, type (1) equivalence relation is not applicable either.
\end{proof}
%

We analyze the equivalence relation $\sim$ of type (1) in more combinatorial ways.
Let $K$ be given. A {\em path $\gamma$} is a word $S_{\epsilon_k}^{\eta_k}\dots S_{\epsilon_1}^{\eta_1}$ of $\{S_0,S_+^{\pm1},S_-^{\pm1}\}$. 
We say that $\gamma$ {\em realizes} a sequence $(L_0, L_1,\dots, L_k)$ in $\mathcal{L}(K)$ if it satisfies the following:
\begin{enumerate}
\item $L_i=L_{i-1}$ if $\epsilon_i=0$;
\item $L_i = S_{\epsilon_i}(L_{i-1})$ if $\epsilon\neq 0, \eta_i=1$; and 
\item $L_{i-1} = S_{\epsilon_i}(L_i)$ if $\epsilon\neq 0, \eta_i=-1$.
\end{enumerate}

A path $\gamma$ is {\em realizable} at $L\in\mathcal{L}(K)$ if $\gamma$ realizes at least 1 sequence $(L_0,\dots,L_k)$ with $L=L_0$.
We denote by $\gamma(L)$ the set consisting of all possible ends of sequences that $\gamma$ realizes. By definition, $\gamma(L)\neq\emptyset$ if and only if $\gamma$ is realizable at $L$.
%
Then it is obvious to check that if two paths $\gamma_1$ and $\gamma_2$ are realizable at $L$ and $L'\in \gamma_1(L)$, respectively, then the concatenation $\gamma_2\cdot\gamma_1$ is also realizable at $L$.

We define the {\em reverse $\bar\gamma=S_{\epsilon_k}^{-\eta_k}\dots S_{\epsilon_1}^{-\eta_1}$} of $\gamma$ 
by changing all exponents. Note that this is different from the usual inverse. 
Geometrically, for any mountain ranges that $\gamma$ and $\bar\gamma$ are realized, $\bar\gamma$ goes down or right if $\gamma$ goes up or left, respectively, and {\em vice versa} because all exponents are reversed. See Figure~\ref{fig:paths} for example.
This operation plays an important role for describing the equivalence relation $\sim$ as follows.

\begin{figure}[ht]
\begingroup%
  \makeatletter%
  \providecommand\color[2][]{%
    \errmessage{(Inkscape) Color is used for the text in Inkscape, but the package 'color.sty' is not loaded}%
    \renewcommand\color[2][]{}%
  }%
  \providecommand\transparent[1]{%
    \errmessage{(Inkscape) Transparency is used (non-zero) for the text in Inkscape, but the package 'transparent.sty' is not loaded}%
    \renewcommand\transparent[1]{}%
  }%
  \providecommand\rotatebox[2]{#2}%
  \ifx\svgwidth\undefined%
    \setlength{\unitlength}{272.34291396bp}%
    \ifx\svgscale\undefined%
      \relax%
    \else%
      \setlength{\unitlength}{\unitlength * \real{\svgscale}}%
    \fi%
  \else%
    \setlength{\unitlength}{\svgwidth}%
  \fi%
  \global\let\svgwidth\undefined%
  \global\let\svgscale\undefined%
  \makeatother%
  \begin{picture}(1,0.1883228)%
    \put(0,0){\includegraphics[width=\unitlength]{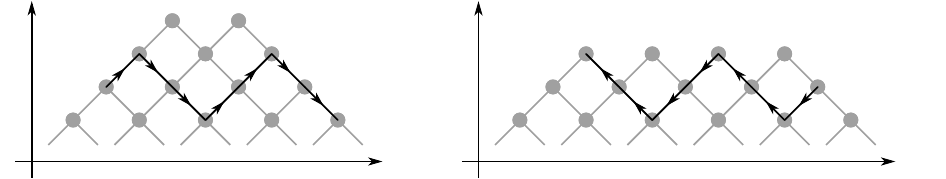}}%
    \put(-0.00145255,0.1750277){\color[rgb]{0,0,0}\makebox(0,0)[lb]{\smash{$tb$}}}%
    \put(0.40977289,0.01753712){\color[rgb]{0,0,0}\makebox(0,0)[lb]{\smash{$r$}}}%
    \put(0.47101922,0.1750277){\color[rgb]{0,0,0}\makebox(0,0)[lb]{\smash{$tb$}}}%
    \put(0.95224044,0.01753712){\color[rgb]{0,0,0}\makebox(0,0)[lb]{\smash{$r$}}}%
    \put(0.10354118,0.12253081){\color[rgb]{0,0,0}\makebox(0,0)[lb]{\smash{$\gamma$}}}%
    \put(0.87162022,0.10565683){\color[rgb]{0,0,0}\makebox(0,0)[lb]{\smash{$\bar\gamma$}}}%
  \end{picture}%
\endgroup%
\caption{A path $\gamma=S_+^2S_-^{-2}S_+^2S_-^{-1}$ and its reverse $\bar\gamma=S_+^{-2}S_-^2S_+^{-2}S_-$}
\label{fig:paths}
\end{figure}

\begin{lem}\label{lem:paths1}
Let $K$ be a prime knot. Then $L_1\#\dots\#L_n=L_1'\#\dots\#L_n'$ in $\mathcal{L}(\#^n K)$ if and only if there exists a set $\{\gamma_1,\dots,\gamma_n\}$ of $n$ paths of length $k$ such that
\begin{enumerate}
\item there exists a permutation $\sigma\in\mathbf{S}_n$ such that $L_{\sigma(i)}'\in \gamma_i(L_i)$,
\item for each $\ell\le k$, $\ell$-th words of $\gamma_i$'s are either 
$[S_+, S_+^{-1}, S_0,\dots,S_0]$ or 
$[S_-, S_-^{-1}, S_0,\dots,S_0]$.
\end{enumerate}
\end{lem}
\begin{proof}
If two elements are equivalent in $\prod^n\mathcal{L}(K)$, then there is a sequence of the equivalence relations of type (1) and (2), which gives a simplicial path $\Gamma:I=[0,k]\to\operatorname{Sym}^n(\mathcal{L}(K))$ joining $[L_1,\dots,L_n]$ and $[L_1',\dots,L_n']$.
Moreover, a map $\Gamma$ can be interpreted as a map $\tilde\Gamma:\tilde I\to \mathcal{L}(K)$ where $\pi:\tilde I\to I$ is an $n$-fold branched simplicial covering map. Note that $\tilde\Gamma$ recovers $\Gamma$ by regarding $\pi^{-1}(t)$ as a set with repetition and sending this set via $\tilde\Gamma$.
This is a standard description for maps into the symmetric product spaces. See Figure~\ref{fig:covering} for an example.

We can resolve the branch locus of $\tilde I$ to obtain $n$ disjoint intervals $nI$ and $n$ paths $\{\gamma_i\}$ of length $k$ in $\mathcal{L}(K)$ such that each joins some of $L_i$ and $L_j'$.
Note that this process is not unique, but the resulting paths satisfy the condition above by definition of the connected sum. Therefore the existence of paths is the same as the equivalence with respect to $\sim$.
\end{proof}

\begin{figure}[ht]
\begingroup%
  \makeatletter%
  \providecommand\color[2][]{%
    \errmessage{(Inkscape) Color is used for the text in Inkscape, but the package 'color.sty' is not loaded}%
    \renewcommand\color[2][]{}%
  }%
  \providecommand\transparent[1]{%
    \errmessage{(Inkscape) Transparency is used (non-zero) for the text in Inkscape, but the package 'transparent.sty' is not loaded}%
    \renewcommand\transparent[1]{}%
  }%
  \providecommand\rotatebox[2]{#2}%
  \ifx\svgwidth\undefined%
    \setlength{\unitlength}{229.98046875bp}%
    \ifx\svgscale\undefined%
      \relax%
    \else%
      \setlength{\unitlength}{\unitlength * \real{\svgscale}}%
    \fi%
  \else%
    \setlength{\unitlength}{\svgwidth}%
  \fi%
  \global\let\svgwidth\undefined%
  \global\let\svgscale\undefined%
  \makeatother%
  \begin{picture}(1,0.49569426)%
    \put(0,0){\includegraphics[width=\unitlength]{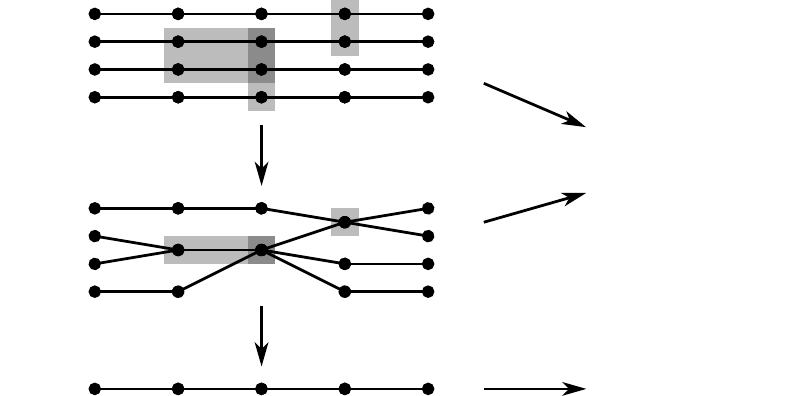}}%
    \put(0.77978769,0.00869638){\color[rgb]{0,0,0}\makebox(0,0)[lb]{\smash{$Sym^n(X)$}}}%
    \put(0.77978769,0.2695881){\color[rgb]{0,0,0}\makebox(0,0)[lb]{\smash{$X$}}}%
    \put(0.01450531,0.00869638){\color[rgb]{0,0,0}\makebox(0,0)[lb]{\smash{$I=$}}}%
    \put(0.01450531,0.18262419){\color[rgb]{0,0,0}\makebox(0,0)[lb]{\smash{$\tilde I =$}}}%
    \put(-0.00288747,0.42612313){\color[rgb]{0,0,0}\makebox(0,0)[lb]{\smash{$nI=$}}}%
    \put(0.65803822,0.02608916){\color[rgb]{0,0,0}\makebox(0,0)[lb]{\smash{$\Gamma$}}}%
    \put(0.65803822,0.18262419){\color[rgb]{0,0,0}\makebox(0,0)[lb]{\smash{$\tilde\Gamma$}}}%
    \put(0.6375396,0.38636822){\color[rgb]{0,0,0}\makebox(0,0)[lb]{\smash{$\{\gamma_i\}$}}}%
  \end{picture}%
\endgroup%
\caption{Multi-paths $\{\gamma_i\}$ describing a path $\Gamma:I\to\operatorname{Sym}^n(X)$ space.}
\label{fig:covering}
\end{figure}

\begin{lem}\label{lem:paths2}
Let $K_1, K_2$ be relatively prime knots, and $L_i, L_i'\in\mathcal{L}(K_i)$ for each $i=1,2$.
Then $L_1\#L_2=L_1'\#L_2'$ in $\mathcal{L}(K_1\#K_2)$ if and only if there exists a path $\gamma$ such that $L_1'\in\gamma(L_1)$ and $L_2'\in\bar\gamma(L_2)$.

In particular, we can conclude that $L_1\#L_2$ is different from $L_1'\#L_2'$ if there is no such path.
\end{lem}
\begin{proof}
Since $K_1$ and $K_2$ are relatively prime, the equivalence relation by permuting components is never applicable. Hence we only need to consider the first type of equivalence relation in Theorem~\ref{thm:connectedsum}.

However, the generator for the equivalence relation of the first type obviously defines a path of length 1 which satisfies the assumption. Moreover, realizable paths are closed under concatenations whenever their ends match, and therefore the lemma follows.
\end{proof}

\section{Main results}

To prove Theorem~\ref{thm:criterion} we will show the following propositions.

\begin{prop}\label{prop:simple1}
Let $K$ be a prime knot. If $\#^n K$ is Legendrian simple, then so is $K$.
\end{prop}

\begin{prop}\label{prop:simple2}
Let $K_1, K_2$ be two relatively prime knots.
If $K_1\# K_2$ is Legendrian simple, then so are $K_1$ and $K_2$.
\end{prop}

Hence for the connected sum to be Legendrian simple, each of its summands must be Legendrian simple.
However, even for Legendrian simple knots, their connected sum need not be Legendrian simple when they have many peaks as follows.

\begin{prop}\label{prop:simplepeaks1}
Let $K$ be a prime and Legendrian simple knot. Then $\#^n K$ is Legendrian simple for $n\ge 2$ if and only if $|Peak(K)|\le 2$.

Moreover,
$$
|Peak(\#^n K)|=\begin{cases}
1 & \text{ if }|Peak(K)|=1;\\
n+1 & \text{ if }|Peak(K)|=2.
\end{cases}
$$
\end{prop}

\begin{prop}\label{prop:simplepeaks2}
Let $K_1$ and $K_2$ be relatively prime and Legendrian simple knots.
Then $K_1\# K_2$ is Legendrian simple if and only if 
either $|Peak(K_1)|=1$ or $|Peak(K_2)|=1$. In this case, 
$$|Peak(K_1\# K_2)| = |Peak(K_1)|\cdot |Peak(K_2)|.$$
\end{prop}

Then Theorem~\ref{thm:criterion} is nothing but the reorganization of the propositions above.

Now we prove the propositions.

\begin{proof}[Proof of Proposition~\ref{prop:simple1}]
Let $L\in Peak(\mathcal{L}(K))$ be a maximal element, and suppose that $K$ is Legendrian nonsimple but $\#^n K$ is Legendrian simple.
Then we can choose $N_{max}\in\mathcal{N}(K)$ 
as before, and there are two cases as follows by Lemma~\ref{lem:nmax}.

If $\Phi^{-1}(N_{max})\cap Peak(\mathcal{L}(K_1))\neq\emptyset$, then there are two different Legendrian knots $L_1, L_1'\in \Phi^{-1}(N_{max})$ such that $L_1\in Peak(\mathcal{L}(K))$. Since $L\in Peak(\mathcal{L}(K))$ as well, $L_1\#(\#^{n-1}L)\in Peak(\mathcal{L}(\#^nK))$ by Corollary~\ref{cor:peaks1}.
If $L_1'\not\in Peak(\mathcal{L}(K))$, then $L_1'\# (\#^{n-1}L)\not\in Peak(\mathcal{L}(\#^nK))$. Therefore $L_1\#(\#^{n-1}L)$ and $L_1'\#(\#^{n-1}L)$ are different.
Otherwise, if $L_1'\in Peak(\mathcal{L}(K))$, then $L_1\#(\#^{n-1}L)$ and $L_1'\#(\#^{n-1}L)$ are still different since $[L_1,L,\dots,L]\neq[L_1',L,\dots,L]$ and by Corollary~\ref{cor:peaks1}.

If $\Phi^{-1}(N_{max})\cap Peak(\mathcal{L}(K))=\emptyset$, then $\Phi^{-1}(N_{max})=\{L_1, L_1'\}$ by lemma~\ref{lem:nmax}.
Suppose $L_1\#(\#^{n-1}L)$ and $L_1'\#(\#^{n-1}L)$ are equivalent. 
Then by Lemma~\ref{lem:paths1}, we may assume that there are paths $\gamma_1,\dots,\gamma_n$ such that $\gamma_1$ is realizable at $L_1$.
Note that all $\gamma_i$'s are lying above $N_{max}$ except for end points if we take $\Phi$.
Since all points above $N_{max}$ are simple, we may identify $\mathcal{L}(K)$ and $\mathcal{M}(K)$ above $N_{max}$.

If $L_1'\not\in\gamma_1(L_1)$, then $L\in\gamma_1(L_1)$ and $L_1'\in\gamma_i(L)$ for some $i\neq 1$. This implies that both $L_1$ and $L_1'$ can be joined with $L$ above $N_{max}$ but this is impossible because parents of $L_1$ and $L_1'$ are lying in the different connected components of the region above $N_{max}$.

On the other hand, if $L_1'\in\gamma_1(L_1)$ then by the exactly same reason about the parents of $L_1$ and $L_1'$, this is impossible too.
%

In all cases, $L_1\#(\#^{n-1}L)$ and $L_1'\#(\#^{n-1}L)$ are different but it is obvious that they share the classical invariants. Therefore this contradicts to the Legendrian simplicity of $\#^nK$.
\end{proof}

\begin{proof}[Proof of Proposition~\ref{prop:simple2}]
The proof is essentially same as the previous one by using Corollary~\ref{cor:peaks2} and Lemma~\ref{lem:paths2} instead of Corollary~\ref{cor:peaks1} and Lemma~\ref{lem:paths1}.
%
%
\end{proof}

\begin{proof}[Proof of Proposition~\ref{prop:simplepeaks1}]
Recall that $|Peak(K)| = |Valley(K)| + 1$ for a Legendrian simple knot $K$.

Suppose $|Peak(K)|\ge 3$. Then there are at least two Legendrian knots $V_1, V_2$ lying in $Valley(K)$.
We may assume that $r(V_1)<r(V_2)$.
For each $V_i$, there are two parents $L_i, L_i'$ so that $tb(L_i)=tb(L_i')=tb(V_i)+1$, and $r(V_1)= r(L_1) + 1 = r(L_1') - 1$, $r(V_2)=r(L_2)-1 = r(L_2')+1$.
In addition, we fix a maximal $L_3\in Peak(K)$ as depicted in Figure~\ref{fig:manypeaks}.

\begin{figure}[ht]
\begingroup%
  \makeatletter%
  \providecommand\color[2][]{%
    \errmessage{(Inkscape) Color is used for the text in Inkscape, but the package 'color.sty' is not loaded}%
    \renewcommand\color[2][]{}%
  }%
  \providecommand\transparent[1]{%
    \errmessage{(Inkscape) Transparency is used (non-zero) for the text in Inkscape, but the package 'transparent.sty' is not loaded}%
    \renewcommand\transparent[1]{}%
  }%
  \providecommand\rotatebox[2]{#2}%
  \ifx\svgwidth\undefined%
    \setlength{\unitlength}{252.15234375bp}%
    \ifx\svgscale\undefined%
      \relax%
    \else%
      \setlength{\unitlength}{\unitlength * \real{\svgscale}}%
    \fi%
  \else%
    \setlength{\unitlength}{\svgwidth}%
  \fi%
  \global\let\svgwidth\undefined%
  \global\let\svgscale\undefined%
  \makeatother%
  \begin{picture}(1,0.38900444)%
    \put(0,0){\includegraphics[width=\unitlength]{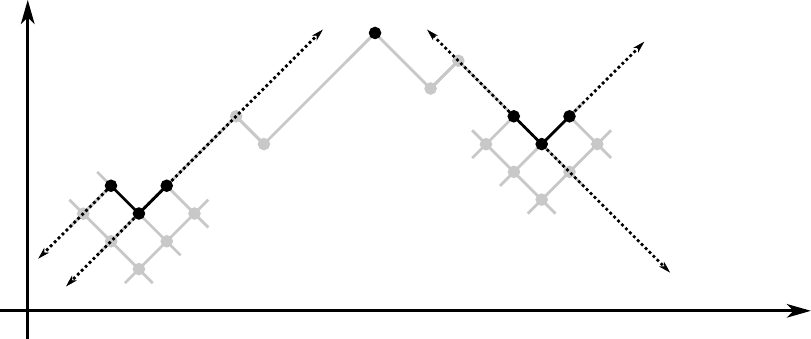}}%
    \put(0.15600068,0.10522122){\color[rgb]{0,0,0}\makebox(0,0)[lb]{\smash{$V_1$}}}%
    \put(0.1084104,0.21040178){\color[rgb]{0,0,0}\makebox(0,0)[lb]{\smash{$L_1$}}}%
    \put(0.11359096,0.07040178){\color[rgb]{0,0,0}\makebox(0,0)[lb]{\smash{$\ell_1$}}}%
    \put(0.17359096,0.21040178){\color[rgb]{0,0,0}\makebox(0,0)[lb]{\smash{$L_1'$}}}%
    \put(0.04359096,0.14040178){\color[rgb]{0,0,0}\makebox(0,0)[lb]{\smash{$\ell_1'$}}}%
    \put(-0.01263358,0.37489946){\color[rgb]{0,0,0}\makebox(0,0)[lb]{\smash{$tb$}}}%
    \put(1,0.03176752){\color[rgb]{0,0,0}\makebox(0,0)[lb]{\smash{$r$}}}%
    \put(0.61844975,0.2897189){\color[rgb]{0,0,0}\makebox(0,0)[lb]{\smash{$L_2'$}}}%
    \put(0.72844975,0.1197189){\color[rgb]{0,0,0}\makebox(0,0)[lb]{\smash{$\ell_2$}}}%
    \put(0.69363031,0.2897189){\color[rgb]{0,0,0}\makebox(0,0)[lb]{\smash{$L_2$}}}%
    \put(0.73363031,0.3497189){\color[rgb]{0,0,0}\makebox(0,0)[lb]{\smash{$\ell_2'$}}}%
    \put(0.40981549,0.38489947){\color[rgb]{0,0,0}\makebox(0,0)[lb]{\smash{$L_3$}}}%
    \put(0.64604003,0.19453835){\color[rgb]{0,0,0}\makebox(0,0)[lb]{\smash{$V_2$}}}%
  \end{picture}%
\endgroup%
\caption{Two valleys $V_i$, hillsides $\ell_i$, parents $L_i, L_i'$ and chosen peak $L_3$}
\label{fig:manypeaks}
\end{figure}
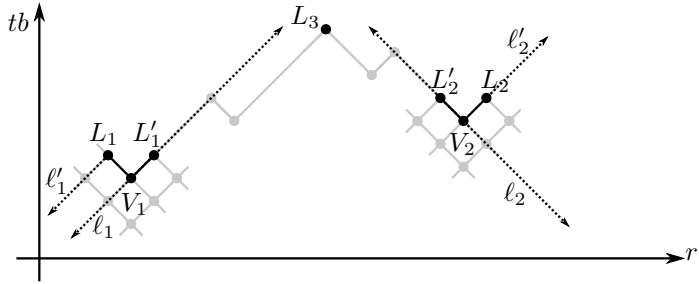

Suppose $L_1\#L_2\#(\#^{n-2}L_3)$ and $L_1'\#L_2'\#(\#^{n-2}L_3)$ are same in $\mathcal{L}(\#^nK)$.
Then there are paths $\gamma_1,\dots,\gamma_n$ as before. We may assume that $\gamma_1$ and $\gamma_2$ are realizable at $L_1$ and $L_2$, respectively.
For each valley $V_i$, a {\em hillside $\ell_i$} is the line defined by
$$
\ell_1: tb-r = tb(V_1)-r(V_1),\quad
\ell_2:tb+r = tb(V_2)+r(V_2).
$$

We claim that each $\gamma_i$ never hit the hillside $\ell_i$.

Suppose not for $\gamma_1$.
Then the last move just before hitting $\ell_1$ must correspond to $S_+$ from a point in the ray 
$$\ell_1':tb-r = tb(L_1)-r(L_1),\quad tb\le tb(L_1).$$
At that time, since $\gamma_1$ lies at the same level of the initial position or below and there is no point above $L_3$, the only possibility is that 
$\gamma_2$ must lie on the another hillside 
$$\ell_2':tb-r = tb(L_2) - r(L_2)$$
at the same level of $L_2$ or above.

By Lemma~\ref{lem:paths1}, the corresponding move in $\gamma_2$ is $S_+^{-1}$ but it can not be performed since there is no point above $\ell_2'$. This is a contradiction.
Similarly $\gamma_2$ never hit the hillside $\ell_2$, and the claim is proved.

This claim implies that $\gamma_1(L_1)$ and $\gamma_2(L_2')$ are separated by 2 lines $\ell_1$ and $\ell_2$, and therefore $L_1', L_2\not\in \gamma_1(L_1)$ and $L_1,L_2\not \in\gamma_2(L_2')$. 

The only possibility is that $L_3\in\gamma_1(L_1)\cap \gamma_2(L_2')$.

$L_1'\not\in\gamma_1(L_1)$ and $L_2\not\in\gamma_2(L_2')$, and therefore only possibility is that $L_3\in\gamma_1(L_1)\cap\gamma_2(L_2')$. However this is not possible either because the regions where $\gamma_i$'s are lying are separated by lines $\ell_1, \ell_2$. 
Therefore $L_1\#L_2'\#(\#^{n-2}L_3)$ and $L_1'\#L_2\#(\#^{n-2}L_3)$ are different but share the classical invariants. Hence $\#^n K$ is not Legendrian simple.

Suppose $|Peak(K)|=1$, that is, $\mathcal{L}(K)$ has the greatest element $L$ and all Legendrian knots in $\mathcal{L}(K)$ are of the form $S_+^aS_-^b(L)$.
Moreover, these stabilizations can be relocated freely among the connected summands. Hence any Legendrian knot in $\mathcal{L}(\#^n K)$ is equivalent to 
$S_+^a S_-^b(L) \# (\#^{n-1}L)$ for some $a,b\ge 0$, and its $tb$ and $r$ determine $a$ and $b$ uniquely. Therefore $\#^nK$ is Legendrian simple.

Finally, suppose $Peak(K)=\{P_1, P_2\}$ and $Valley(K)=\{V\}$. Without loss of generality, we may assume that $r(P_1)<r(V)<r(P_2)$. 
Then 
$$S_+^{-r'(P_1)} (P_1) = V = S_-^{r'(P_2)}(P_2),\text{ and } tb'(P_1) = -r'(P_1), tb'(P_2) = r'(P_2)$$
where $r'(L) = r(L) - r(V)$ and $tb'(L) = tb(L)-tb(V)$. See Figure~\ref{fig:2peaks}.

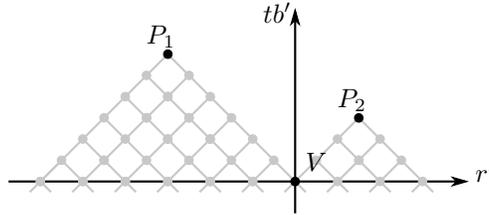
\begin{figure}[ht]
\begingroup%
  \makeatletter%
  \providecommand\color[2][]{%
    \errmessage{(Inkscape) Color is used for the text in Inkscape, but the package 'color.sty' is not loaded}%
    \renewcommand\color[2][]{}%
  }%
  \providecommand\transparent[1]{%
    \errmessage{(Inkscape) Transparency is used (non-zero) for the text in Inkscape, but the package 'transparent.sty' is not loaded}%
    \renewcommand\transparent[1]{}%
  }%
  \providecommand\rotatebox[2]{#2}%
  \ifx\svgwidth\undefined%
    \setlength{\unitlength}{191.01953125bp}%
    \ifx\svgscale\undefined%
      \relax%
    \else%
      \setlength{\unitlength}{\unitlength * \real{\svgscale}}%
    \fi%
  \else%
    \setlength{\unitlength}{\svgwidth}%
  \fi%
  \global\let\svgwidth\undefined%
  \global\let\svgscale\undefined%
  \makeatother%
  \begin{picture}(1,0.40879788)%
    \put(0,0){\includegraphics[width=\unitlength]{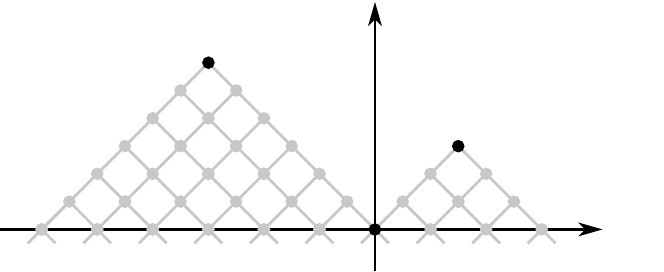}}%
    \put(0.58632748,0.08381475){\color[rgb]{0,0,0}\makebox(0,0)[lb]{\smash{$V$}}}%
    \put(0.50256641,0.37697849){\color[rgb]{0,0,0}\makebox(0,0)[lb]{\smash{$tb'$}}}%
    \put(0.92137175,0.06287448){\color[rgb]{0,0,0}\makebox(0,0)[lb]{\smash{$r'$}}}%
    \put(0.27222347,0.33509795){\color[rgb]{0,0,0}\makebox(0,0)[lb]{\smash{$P_1$}}}%
    \put(0.64914828,0.20945635){\color[rgb]{0,0,0}\makebox(0,0)[lb]{\smash{$P_2$}}}%
  \end{picture}%
\endgroup%
\caption{Mountain range with exactly 2 peaks}
\label{fig:2peaks}
\end{figure}

As before, all Legendrian knots in $\mathcal{L}(K)$ are of the form either $S_+^a S_-^b (P_1)$ or $S_+^a S_-^b (P_2)$, and any Legendrian knot $L$ in $\mathcal{L}(\#^n K)$ is equivalent to 
$S_+^a S_-^b\left(( \#^p P_1) \# (\#^q P_2)\right)$ for $p+q=n$.
We simply denote this by $L(a,b,p,q)$. Then for $L=L(a,b,p,q)$,
$$tb(L)= p\cdot tb(P_1) + q\cdot tb(P_2) + (n-1) - a -b,$$
$$r(L)= p\cdot r(P_1) + q\cdot r(P_2) + a-b.$$


We consider two invariants defined by using $tb$ and $r$ as follows.
$$X(L)=\frac12\left(tb(L) + r(L) - n(tb(V)+r(V)) - (n-1)\right)=q\cdot r'(P_2)-b,$$
$$Y(L)=-\frac12\left(tb(L) - r(L) -n(tb(V)- r(V)) - (n-1)\right)=p\cdot r'(P_1)+a.$$

Now suppose two Legendrian knots $L=L(a,b,p,q)$ and $L'=L(a',b',p',q')$ share the same $tb$ and $r$.
Then they also share $X$ and $Y$, and so
$$
X(L) - X(L') = (q-q') r'(P_2)- (b-b') = 0,
$$
$$
Y(L)-Y(L')=(p-p') r'(P_1)+(a-a') = 0. 
$$

If $b=b'$ or $a=a'$ then $p=p'$ and $q=q'$ since neither $r'(P_1)$ nor $r'(P_2)$ vanishes.

If $b>b'$, then $q>q'$, $p<p'$, and $$a'=a+(p-p') r'(P_1),\quad b=(q-q')r'(P_2)+b'.$$
Moreover, 
\begin{align*}
L(a,b,p,q) &=S_+^aS_-^b( (\#^p P_1)\#(\#^q P_2))\\
&=S_+^a S_-^{b'}\left( (\#^p P_1)\# (\#^{q'}P_2) \# (\#^{q-q'} S_-^{r'(P_2)} P_2)\right)\\
&=S_+^a S_-^{b'}\left( (\#^p P_1)\# (\#^{q'}P_2) \# (\#^{q-q'} V)\right)\\
&=S_+^a S_-^{b'}\left( (\#^p P_1)\# (\#^{q'}P_2) \# (\#^{p'-p} S_+^{-r'(P_1)}P_1)\right)\\
&=S_+^{a+(p-p')r'(P_1)} S_-^{b'}\left( (\#^{p'} P_1)\# (\#^{q'}P_2)\right)\\
&=S_+^{a'} S_-^{b'}\left( (\#^{p'} P_1)\# (\#^{q'}P_2)\right)=L(a',b',p',q').
\end{align*}

Conversely, the same result holds for $b<b'$ by changing the roles of $a, b,p,q$ and $a', b',p',q'$. Therefore the classical invariants determine exactly one Legendrian knot in $\mathcal{L}(\#^n K)$, and so $\#^n K$ is Legendrian simple.

The number of peaks directly follows from Corollary~\ref{cor:peaks1}.
\end{proof}

\begin{proof}[Proof of Proposition~\ref{prop:simplepeaks2}]
Suppose $|Peak(K_i)|\ge 2$, or $|Valley(K_i)|\ge 1$ for all $i$.
Let $V_i\in Valley(K_i)$ and $L_i, L_i'$ be two stabilizations of $V_i$ such that $r(V_i) = r(L_i) + 1 = r(L_i') -1$.
Then $L_1\#L_2'$ and $L_1'\#L_2$ are different by the essentially same argument as before and by Lemma~\ref{lem:paths2}. Therefore $K_1\#K_2$ is not Legendrian simple.

Suppose $|Peak(K_1)|=1$, and let $L$ be the unique maximal element of $\mathcal{L}(K_1)$. Then as before, all Legendrian knots in $\mathcal{L}(K_1)$ are of the form $S_+^aS_-^b(L)$, and so any Legendrian knot in $\mathcal{L}(K_1\#K_2)$ is equivalent to $L\# L_2$ for some $L_2\in\mathcal{L}(K_2)$ by moving all stabilizations to the second summand.
Therefore, the classical invariants for $L\#L_2$ determine not only the classical invariants for $L_2$, but also $L_2$ itself since $K_2$ is Legendrian simple. This implies the Legendrian simplicity for $K_1\#K_2$.

The number of peaks directly follows from Corollary~\ref{cor:peaks2}.
\end{proof}


\end{document}